\documentclass[12pt,twoside]{amsart}          
\usepackage[utf8]{inputenc}
\usepackage[T1]{fontenc}
\usepackage{amsmath, amssymb, amsthm,enumitem}            
\usepackage{amstext, amsfonts, a4}

 \usepackage[colorlinks=true]{hyperref}
\hypersetup{urlcolor=blue, citecolor=red}

\usepackage{color,transparent}
\usepackage{graphics,graphicx,subfigure}

\usepackage{tree-dvips}
\usepackage{qtree}

\usepackage{times} 

\usepackage{graphics,graphicx,subfigure}
\newcommand{\gr}{Grothendieck ring }

\newcommand{\frml }{formula }

\newcommand{\frmlsf}{formulas}

\newcommand{\eqf}{quantifier elimination}

\newcommand{\N}{ \mathbb{N}}
\newcommand{\Z}{ \mathbb{Z}}

\newcommand{\st}{such that }

\newcommand{\elmts}{ elements }
\newcommand{\elmt}{ element }

\newcommand{\G}{Grothendieck }

\newcommand{\ens}{ set }
\newcommand{\enss}{ sets }

\newcommand{\ess}{simple sets }

\newcommand{\dd}{definable }
\newcommand{ \sse}{ simple sets }

\theoremstyle{plain}
    \newtheorem{teo}{Theorem}[section]
    \newtheorem{lem}[teo]{Lemma}

    \newtheorem{de}[teo]{Definition}
     \newtheorem{cor}[teo]{Corollary}
    \newtheorem{rem}[teo]{Remark} 
    \newtheorem{ex}[teo]{Example} 
\usepackage{comment}
\usepackage{float} 
\def\ogg~{{\rm \og}}   

\def\emptyset{\varnothing}

\makeatletter
\def\BState{\State\hskip-\ALG@thistlm}
\makeatother

\title{Grothendieck ring of the pairing function without cycles}

\address{MSRI, 17 Gauss Way, Berkeley, CA 94720, États-Unis}
\email{esther.elbaz.saban@normalesup.org}

\author{Esther Elbaz}

\date{June 11, 2020}
\begin{document}
\begin{abstract}
A bijection $(l,r)$ between $M^2$ and $M$ is said to be a pairing function with no cycles, if any composition of its coordinate functions has no fixed point.
We compute here the Grothendieck ring of the pairing function without cycles to be isomorphic to $\Z^2\simeq \Z[X]/(X-X^2)$. More generally, for any $n\in \N^*$ and any bijetion without cycles betwen $M$ and $M^n$, the exact same method proves that $K_0(M)=\Z[X]/(X-X^n)$. 

\end{abstract}

\maketitle

\section{Introduction}\label{s.introduction}  A bijection, $(l,r): M \rightarrow M^2; x\mapsto (l(x), r(x))$, between a set and its square, is called a pairing function. It is said to have no cycles if any function obtained as a composition of $l$ and $r$ have no fixed point.

Sets endowed with pairing
functions are also known as Jonsson-Tarski algebras or Cantor
algebras, they have been studied in universal algebra (for instance in
\cite{JT,AS}), and admit rich groups of automorphisms \cite{H}. They have been studied by several authors in a model theoretic perspective. One of the motivation being that they are a simple example of a theory that is stable and cannot be obtained as the limit of superstable theories. As such, they can help understanding super-stable theories by providing counterexamples to their properties.
Pairing function without cycles admit quantifier elimination in the language $\lbrace l, r\rbrace$ and is complete (\cite{PB}).

In this article, we compute the Grothendieck ring of pairing function without cycles. 

The \gr of a structure is obtained by identifying definable sets that are in definable bijection, with the ring structure being induced by laws reflecting respectively the disjoint union (for the additive law) and the Cartesian product (for the multiplicative law).

More precisely:

\begin{de} Let $M$ be a structure. Let us consider $\operatorname{Def}(M)$ the set of subsets in the disjoint union of all finite Cartesian powers of $M$ that are definable with parameters. Let $\sim$ be the equivalent relation defined on $\operatorname{Def}(M)$ by $A\sim B$ if and only if there exists a definable bijection between $A$ and $B$. For any $A\in \operatorname{Def}(M)$, we denote by $[A]$ the class of $A$. 

We define two laws on the quotient set $ \operatorname{Def}(M) $ :
\begin{itemize}
\item An additive law that corresponds to the disjoint union: $[A]+[B]=[A\cap B]+[A\cup B]$.
\item A multiplicative law that corresponds to the Cartesian product: $[A]\times[B]=[A\times B]$.
\end{itemize}
These two laws are compatible with the equivalence relation and the set $ \operatorname{Def}(M) $ endowed with these two laws is the Grothendieck semi-ring of $M$.

\end{de}

Beware that the additive law is not cancelative: $a+b=a+c$ doesn't imply $b=c$. To remedy this, let us consider the equivalence relation defined on $ \operatorname{Def}(M) $ by: $a, b\in \operatorname{Def}(M) $ are equivalent if and only if there exists $c\in \operatorname{Def}(M) $ such that $a+c=b+c$. The quotient set, denoted $\widetilde{\operatorname{Def}}(M)$, is a cancelative monoid for the additive law. Hence, there exists a unique ring (up to isomorphism over $(\widetilde{\operatorname{Def}}(M), +, \times) $) that embeds this last quotient and that is minimal for this property.

\begin{de}
The minimal ring that embeds $(\widetilde{\operatorname{Def}}, +, \times) $ is called the \textbf{Grothendieck ring of} $M$ and is denoted $K_0(M)$.
\end{de}

This ring has been defined by T. Scanlon and J. Krajicek in \cite{SK} who built up a dictionary between the combinatorial properties of a structure and the algebraic properties of its \G ring. For example, a structure admits a zero \G ring, if and only if, there exists a definable set in definable bijection with itself deprived of a point.


This last property called Onto Pigeon Hole principle and abreged Onto-PHP has been used to prove the triviality of \G rings in the case of several valued fields (for example $\Z$-valued fields \cite{HC}). 

In general case, two elementarily equivalent structures do not
necessarily have isomorphic Grothendieck ring. Nevertheless, we prove
here that all pairing function without cycles admit $\Z^2=Z[X]/(X-X^2)$
as their Grothendieck ring. To prove this, we will highlight special class of definable sets that we will call simple sets and that are definable by positive conjunction of positive atomic formulas. Every definable set being a Boolean combination of simple sets, the \gr is generated by their classes. It will prove that they are each either isomorphic to $M$ or to a singleton implying, because of the relation $[M]=[M]^2$, that the \gr is naturally a quotient of $Z[X]/(X-X^2)$. It will remain to show that there is no further relation which will be done by examining the action of definable injection on definable sets.  

The following more general theorem can be proved in exactly the same way:

\begin{teo}
Let $n\in \N^*$ and let $L:=\lbrace l_1,\ldots, l_n\rbrace$ be a language consisting of $n$ symbols of functions. Let $M$ be a $L$-structure such that $(l_1,\ldots, l_n): M \rightarrow M^n$ is a bijection. Assume furthermore that for every term $t$ in $\lbrace l_1,\ldots, l_n\rbrace$, $M\models t(x)\neq x$.
Then $K_0(M)=\Z[X]/(X-X^n)$.
\end{teo}

To make the redaction simple, we restrict ourselves to the case $n=2$.

\section{Representation of formulas and definable sets by binary trees}

To any element we can associate an infinite binary tree: the left child of a node corresponds to its image by the function $l$ and the right child corresponds to its image by the function $r$.\\
We can also use binary trees to represent certain definable sets and definable functions. We explain how in this section.

\begin{de} 
Let $T$ be an infinite and complete binary tree.
We define the \textbf{depth} of a node recursively:
\begin{itemize}
\item by convention, the root is of depth 0;
\item if a node is of depth $p$, then its children are of
depth $p+1$.
\end{itemize}

There is thus one node at depth 0, two at depth
$1, \ldots, $ and $2^p$ at depth $p$. So we count $2^{p+1}-1$ nodes at depth less than or equal to $p$.
\end{de}


\begin{de}{\textbf{Labelling of the nodes}} \\
Let $T$ be an infinite binary tree.
We label the nodes of $T$ by associating them with a finite sequence
of elements of $\lbrace l, r \rbrace$ as follows:
\begin{itemize}
\item the root is labelled by the empty sequence;
\item if a node is labelled by a sequence $\lbrace u_1, \ldots, u_n
\rbrace $, then his left child is labelled by $\lbrace
u_1, \ldots, u_n, l\rbrace $ and his right child by $\lbrace
u_1, \ldots, u_n, r\rbrace$.
\end{itemize}
\end{de}
The nodes of depth $p$ are thus labelled by a sequence of
$p$ elements.

\begin{rem} 
We will denote $T_l$ (respectively $T_r$) the left subtree
(respectively right subtree) of $T$, that is the full subtree of $T$ of root $l(x)$
(respectively of root $r(x)$) where $x$ is the root of $T$.
\end{rem}

In the following, all trees considered will be binary infinite trees.

\begin{de}{\textbf{Tree associated to an element}} 
Let $ x $ be an element of $ M $. \\We associate to $ x $ a binary tree labelled
by elements of $ M $ as follows:
\begin{itemize}
\item its root is labeled by $ x $;
\item if a node is labeled by $ y $, then its left child is labeled
by $l(y) $ and its right child by $r(y)$.
\end{itemize}

We will denote this tree $ T_x $.
\\
The nodes of $ T_x $ are enumerated by finite sequences of elements of
$ \lbrace l, r \rbrace $ as explained above.
\\
If $ i $ is a finite sequence of $ \lbrace l, r \rbrace$, we will denote $ x_i $ the
term that labels the $ i $-th node of $ T_x $.
\\
This notation is consistent with the one proposed above: $ l(x) $
which is denoted $ x_l $ corresponds to the left child of the tree of root
$x $; and $ r (x) $ denoted $ x_r $ corresponds to its right child.
\\
The nodes of $ T_x $ correspond exactly to the set of one variable-terms in $ x $.
\end{de}

\begin{rem} 
The fact that the pairing function has no cycles is equivalent to saying that for every $x\in M$ there is no branch of $T_x$ where an element of $M$ appears more than once.
\end{rem}

\begin{de}{\textbf{Depth of a formula, a set, a
definable function}} 
Let $ \psi $ be a formula.
We call \textbf{depth in $ x $} of $ \psi $ the largest integer $ n $
such that $ \psi $ involves a term in $ x $ corresponding to a node of depth $ n $
in the tree $ T_x $.
\\
Let $ A $ be a definable set. We call \textbf{depth in $ x $}
of $ A $ the smallest integer $ n $ such that there is a formula defining
$ A $ whose depth in $x$ is $ n $. \\
Let $ h $ be a definable application. We call \textbf{depth in
$ x $} of $ h $ the smallest integer $ n $ such that there is a formula
defining the graph of $ h $ whose depth in $x$ is $ n $. 
\end{de}

\section{Primitive formula}
Thanks to \eqf, any formula is equivalent to a Boolean combination of atomic formulas.

\begin{de}{\textbf{Primitive formula}} 
Let us consider a formula $\phi$ with possibly several free variables. If $\phi$ is a conjunction of equalities and inequalities between terms, then we say that $\phi$ is \textbf{primitive}. 
\\
If $\phi$ is a conjunction of equalities between terms, then we say that $\phi$ is
\textbf{positive primitive}.
\end{de}

\begin{de}{\textbf{Tree associated to a primitive formula with one variable}} 
\begin{enumerate}[label=(\alph*)]
\item \label{item_a} Let $\psi(x)=\wedge_j \phi_j(x)$ be a primitive formula with a single free variable $x$. We can associate to it a binary tree, denoted $T(\psi)$, such that for every finite sequence of $\lbrace l, r\rbrace$, $i$, the $i$-th node of $T$ is
labelled by the list of formulas that involve the term
corresponding to $x_i$.
By abuse, we can also say that this tree is associated with
$\psi(M)$ or that it defines $\psi(M)$.
\item  If $F(x, y)$ is a primitive formula with $2$ variables,
we associate to $F$ the\textbf{ tree in $x$} constructed as in \ref{item_a} while considering $y$ as a parameter.
\end{enumerate}
\end{de}

\begin{rem} 
The association of a tree to such formulas is a bijection. But two equivalent formulas can have distinct trees.
\end{rem}

\begin{ex} 
1. Let $c\in M$. Consider the formula $\psi(x):=(l(x)=c) \wedge
(l(x)=r(x))$.
The tree associated with $\psi(x)$ is the tree, $T$, whose root has a left child labelled by $\lbrace l(x)=r(x), l(x)=c \rbrace $ and a right child labelled by $\lbrace l(x)=r(x)\rbrace$. The other nodes
are labelled by the empty list.

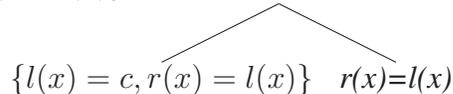
\Tree[ {$\lbrace l(x)=c, r(x)=l(x)\rbrace$} r(x)=l(x) ].

The formula associated with the tree $T$ is the formula  $\phi(T)(x)=(l(x)=c) \wedge
(l(x)=r(x)) \wedge (l(x)=c)$.\\
The formulas $\psi(x)$ and  $\phi(T)(x)$ are equivalent.

\medskip
\noindent
2. Let $c, c'\in M$ and $\psi(x, y)$ be the formula $ (l(x)=c) \wedge
(r(x)=l(y))$. \\
The tree in $x$ associated with this formula is the tree whose root has a left child labelled by $\lbrace l(x)=c \rbrace $, and a right child labelled by $\lbrace (r(x)=l(y)) \rbrace$.

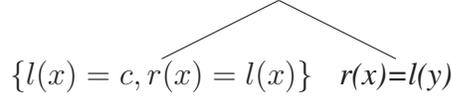
\Tree[ {$\lbrace l(x)=c, r(x)=l(x)\rbrace$} r(x)=l(y) ].
\end{ex}

\begin{rem} 
Let $T$ be a tree labelled as in \ref{treef}.
Suppose the root of $T$ is labelled by the empty list, then $\phi(T)(x)$ is equivalent to
$\phi (T_l) (l(x)) \wedge \phi(T_r) (r(x))$.
\\
Since $\Theta$ is bijective, for every $y\in \phi(T_l) (M)$ and every
$z \in \phi(T_r)(M)$,
the tree $T'$ such that $T'_l=T_y$, $T'_r=T_z$, and whose root is
labelled by the empty list,
is the tree of an element of $\phi(T) (M)$.
\end{rem}

Reciprocally we can associate a one-variable primitive formula to certain trees.

\begin{de}\label{treef}{\textbf{Primitive formula associated with a tree}}  
Let $T$ be a binary tree whose nodes are labelled by formulas. 
Let us suppose two things:
\begin{enumerate}
\item For every \textbf{$i$}, finite sequence of $\lbrace l, r\rbrace$, the \textbf{$i$}-th node of $T$ is labelled with a finite list of equalities or inequalities between $x_i$ and other terms in $x$ (possibly constant), that is by a list of formulas of the form: 
\begin{itemize}
\item $t(x)=c $ 
\item $t(x)=t'(x)$
\item $t(x) \neq c$
\item $t(x) \neq t'(x)$
\end{itemize}
where $t(x)$ is the term corresponding to $x_i $, $t'(x)$ is a term having $x$ as
a single variable and $c$ is a constant.
\item  Only a finite number of nodes of $T$ are labelled by a non-empty list.
\end{enumerate}

\medskip
\noindent
We can associate with $T$ the formula obtained as the conjunction of all the formulas that appear in the labels of $T$.
\\
We denote $\phi(T)(x)$ this (primitive) formula.
\end{de}

\section{Closed tree and simple formulas}

\begin{de} \label{closed_tree}
\begin{enumerate}[label=(\alph*)]
\item 
Let $\psi$ be a primitive formula and $T$ be its associated tree.
We say that $T$ (or by abuse $\psi$) is \textbf{closed} if there exist $f_1, \ldots, f_n$, nodes of $T$ with non-empty label that satisfy the following conditions:
\begin{itemize}
\item \label{item_adef} every node of $T$ is either a descendant or an antecedent of one of the node $f_1, \ldots, f_n$;
\item for every $i\neq j$, $f_i$ is not a descendant of $ f_j$.
\end{itemize}

\item Let $T$ be a closed tree and $f_1, \ldots, f_n$ be nodes of $T$ satisfying the properties of \ref{item_adef}.
Let $T'$ be the tree obtained from $T$ by removing the label of all its nodes except $f_1, \ldots, f_n$. We say that $T'$ is a \textbf{skeleton} tree of $T$.

\item Let $T$ be a tree associated with a primitive formula $\psi$. We say that $T'$ is a \textbf{closed subtree} of $T$, if $T'$ is closed and if $T'$ is obtained from $T$ by removing the labels of some of its nodes.
\end{enumerate}
\end{de}

\begin{rem} 
Let $T$ be a closed subtree and  $f_1, \ldots, f_n$ be nodes of $T$ with a non-empty label that satisfy the condition of the definition \ref{closed_tree}. Let $T'$ be a skeleton tree of $T$. Then $\phi(T')$ is a formula equivalent to $\phi(T)$.
\end{rem}

\begin{ex} 
Here is an example of a closed tree:
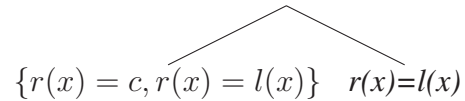
\Tree[ {$\lbrace r(x)=c, r(x)=l(x)\rbrace$} r(x)=l(x) ].

                  \end{ex}

\begin{lem} \label{yyy} 
Let $\psi(x)$ be a primitive formula obtained as a conjunction of formulas of the form $t(x)=c $ where $t$ is a non-constant term and $c$ is a constant.
Suppose moreover that $\psi(x)$ is closed.
Then $\psi$ defines a singleton.
\end{lem}
\begin{proof} 
This is an immediate consequence of the fact that $\Theta$ is a
bijection and it can be proved by recurrence on the depth of
$\psi$,  $ q $. \\
For $ q=1 $, $\psi(x)$ is of the form $(l(x)=c)\wedge (r(x)=c')$ where $c, c'\in M$. Since $\Theta$ is a bijection, $\psi$  defines a singleton.

\medskip
\noindent
Suppose that the property is satisfied for $q \in \mathbb{N}^*$. Let us show
that it is true for $ q+1$.
\\
Let $T$ be the tree associated with $\psi$.
\\
Since $q+1$ is strictly greater than $ 0 $, and $T$ is closed, its root is labelled by the empty list, and
$\psi(x)$ is equivalent to $\phi(T_l) (l(x)) \wedge \phi(T_r) (r(x))$.
\\
The formulas
$\phi(T_l)$ and $\phi(T_r)$ are also closed and, by hypothesis of induction, they define a singleton. Hence $\phi(T)$ also defines a singleton.
\end{proof}

\begin{lem}\label{closed} 
Let $\psi(x)$ be a primitive formula that defines a finite and non-empty set. Then:
\begin{enumerate}[label=(\alph*)]
\item \label{item_alem} The tree associated to $\psi$ admits a closed subtree.
\item  \label{item_blem} The formula $\psi(x)$ is equivalent to a primitive positive \frml and $\psi$ defines in fact a singleton. 
More precisely,
let $T$ be a skeleton of $T(\psi)$ (which is defined as $T(\psi)$ is closed) and $T'$ be the tree obtained from $T$ by removing from the labels of $T$ every formula that is an inequality.
Then $\psi(x)$ is equivalent to $\phi(T')(x)$.
\end{enumerate}
\end{lem}

\begin{proof} 
\ref{item_alem} Let assume that $T(\psi)$ does not admit a closed subtree. 
Then there exists a node of $T(\psi)$ that is neither a descendant nor a ascendent of any node of $T(\psi)$ that is labelled by the empty list. Therefore this node is totally independent of $\psi$.
\\
Let $i$ be a finite sequence of elements of $\lbrace l, r\rbrace$ such that this node corresponds to the term $x_i$.
Let $x\in \psi(M)$ and $z\in M$. Consider the\elmt $y$ of $M$ such that for any term $y_k$ that has not a filiation link with $x_i$, $y_k=x_k$ and such that $y_i=z$. This element satisfies the formula $\psi$ since $x$ does. As this is true for any element $z$, $\psi(M)$ is actually infinite. This contradiction proves that $T(\psi)$ admits a closed subtree.

\medskip
\noindent
\ref{item_blem} Let $T$ be a skeleton tree of $T(\psi)$. Let us suppose, ad absurdio, that the formula associated with $T$ is not positive. 
Let $T'$ be the tree obtained from $T$ by removing from the labels of $T$ every formula that is an inequality.

Two cases are possible: either one of the labels of $T$ that were not empty is now empty in $T'$, or none of the labels of $T'$ that were not empty is now empty.

Let us first show that every non-empty label of $T$ that contains an inequality, also contains an equality. 
Let $t_1, \ldots, t_n$ be the terms corresponding to the nodes of $T$ with non-empty label. Assume $t_n$ corresponds to a node whose label contains only inequalities. For every $x\in M$, let $S_x$ be the set of elements of $M$ that satisfy this conjunction of inequalities imposed on $t_n(x)$. As $M$ is infinite, so is $S$.
It contredicts that $\psi(M)$ is finite.

Now, let $n$ be a node of $T$ whose label, $L_n$, isn't empty. By what precceds it contains an equality. It is easy to see that $T_n$ the tree obtained from $T$ by removing any inequality from $L_n$ is associated with a formula equivalent to $\psi$.
By doing this for all non-empty label, we obtain $T'$ a tree such that $\phi(T')(x)$ is primitive positive and such that $\psi(x)$ is equivalent to $\phi(T')(x)$.

Let us now prove that $\phi(T')$ defines a singleton. Assume it does not. Then by lemma \ref{yyy}, $\phi(T')$ is not equivalent to a conjunction of equalities between terms and constants. Hence, there exists $I$ a set of finite sequence of $\lbrace l, r\rbrace$ such that 
\begin{itemize}
\item for any $i, j\in I$, any $x\in \phi(T')(x)$, $x_i=x_j$,
\item for any constant $c\in M$, there exists $x\in \phi(T')$ such that for any $i\in I$, $x_i\neq c$.
\end{itemize}
It is hence obvious that $\phi(T')$ defines an infinite set which contradicts our hypothesis.
\end{proof}

\begin{de} 
A primitive formula such that there is no filiation links between the terms that appear in it is called \textbf{extended}.

If $\psi(x,y_1,\ldots, y_n)$ is a primitive formula, then we say it is \textbf{extended in $x$}, if for any constants $c_1,\ldots, c_n$, the formula in one variable, $\psi(x, c_1,\ldots, c_n)$, is extended.
\end{de}

\begin{rem}
A closed formula is extended.
\end{rem}

\begin{lem} \label{equivalent_extended}
Let $F(x,y)$ be a satisfiable primitive formula in $y$.
Then there exists $G(x,y)$ such that $G(x,y)$ is extended in $x$ and $G(x,y)$ is equivalent to $F(x,y)$.
\end{lem}

\begin{proof} Let us assume that $F(x,y)$ is not extended in $x$.
Let $t(x)$ and $t'(x)$ be two terms that appear in $F(x,y)$ with filiation links. Let us for instance assume that $t(x)=t''\circ t'(x)$ where there exist $k\in \N^*$ and $t_1, \ldots, t_k \in \lbrace l, r\rbrace$ such that $t''=t_1\circ\ldots\circ t_k$.
Then any atomic formula on $t'(x)$ is equivalent to a conjunction of atomic formula on the terms $\tilde{t}\circ t'(x)$ where $\tilde{t}$ ranges over the set $S:=\lbrace t'_1\circ\ldots\circ t'_k | t'_1, \ldots, t'_k \in \lbrace l, r\rbrace\rbrace$.

There is no filiation link between $t(x)$ and the terms $\tilde{t}\circ t'(x)$ where $\tilde{t}$ ranges over $S$.

By iterating as many times as necessary, we obtain a formula extended in $x$ and equivalent to $F(x,y)$.
\end{proof}

\begin{lem} \label{defsingl} 
Let $l<n$ be two integers and $\psi(x,y_1,\ldots, y_l, y_{1+1},\ldots, y_{1+n})$ be a formula extended in $x$ of the form $(\bigwedge_{j=1}^l t_j(x)=y_j) \wedge (\bigwedge_{j=1}^{n} t_j (x)\neq y_{1+j})\wedge F(x)$ where $F(x)$ is a conjunction of equalities and inequalities between non-constants terms in $x$.
\\
If there exists $c_1,\ldots, c_l, c'_1,\ldots, c'_n$ such that $\psi(x,c_1,\ldots, c_l, c'_1,\ldots, c'_n)$ 
defines a finite set, then \\
$\psi(x,y_1,\ldots, y_l, y'_1,\ldots, y'_n)$ is 
equivalent to $\wedge_{j=1}^l t_j (x)=y_j$.
\\
Moreover, for every $d_1,\ldots, d_l, d'_1,\ldots, d'_n$,\elmts of $M$, $\psi(x,d_1,\ldots, d_l, d'_1,\ldots, d'_n)$ defines either a singleton or the empty set.
\end{lem}

\begin{proof} 
Let $c_1,\ldots, c_l, c'_1,\ldots, c'_n$ such that $\psi(x,c_1,\ldots, c_l, c'_1,\ldots, c'_n)$ defines a finite set. Let $T$ be the skeleton of $T(\psi(x,c_1,\ldots, c_l, c'_1,\ldots, c'_n))$ and $T''$ be the tree obtained from $T$ by removing off the labels of $T'$ every formula that is an inequality. 
By lemma \ref{closed}, $\psi(x,c_1,\ldots, c_l, c'_1,\ldots, c'_n)$ is equivalent to $\phi(T)(x)$. This formula is obviously extended.
It is easy to check that if this formula isn't a conjunction between terms and constants, it then defines an infinite set.
That implies that $\psi(x,c_1,\ldots, c_l, c'_1,\ldots, c'_n)$ is in fact equivalent to $\wedge_{j=1}^l t_j(y)=c_j$.
\\
Since $\psi(x,c_1,\ldots, c_l, c'_1,\ldots, c'_n)$ defines a finite set, its associated tree admits a closed subtree. Because  $\psi(x,c_1,\ldots, c_l, c'_1,\ldots, c'_n)$ is assumed to be extended, that implies that $\wedge_{j=1}^l t_j(y)=c_j$ is closed. Hence, for every $d_1,\ldots, d_l, d'_1,\ldots, d'_n$,\elmts of $M$, $\psi(x,d_1,\ldots, d_l, d'_1,\ldots, d'_n)$  is 
equivalent to the closed formula $\wedge_{j=1}^l t_j (x)=d_j$, which, by lemma \ref{yyy} defines singleton.
\end{proof}

\section{Form of the definable injections}

\begin{lem}{\textbf{Form of the definable injections}} \label{formidable} 
Let $A$ be a definable subset of $M$ and $ h$ be an injective definable function from $A$ to
$A$.
The graph of $ h$ can be defined by a formula of the form $\vee_j
(\phi_j(x) \wedge \phi'_j(x, y))$ where:
\begin{enumerate}
\item every $\phi_j(x)$ is primitive in $x$,
\item the sets $\phi_j (M)$ are all disjoint,
\item for every $j$, $\phi'_j(x, y)$ is a disjunction of positive primitives \frmlsf.
\end{enumerate}
\end{lem}

\begin{proof} 
By elimination of the quantifiers, we can assume that we have
partitioned $A$ into a disjoint union of sets $ A_j$ defined by formulas
$\phi_j$ such that points 1 and 2 are satisfied. Let $\phi'_j(x, y)$ be a formula that defines the graph of $ h$ restricted to $ A_j$. Up to refining the partition $\cup_i A_i$, we can assume that $\phi'_j(x, y)$ is a disjunction of primitives \frmlsf.
\\
By lemma \ref{equivalent_extended}, since $\phi'_j(x, y)$ is satisfiable, it is equivalent to an
extended formulas in $y$. As for every $x$, the formula in $y$, $\phi'(x,y)$ defines a singleton, lemma \ref{defsingl} implies that there are equivalent to a positive primitive formula. 
\end{proof}

\begin{de} 
Let $h$ be a definable injection from a definable set of $M$ into $M$.
We say that a formula, $\phi(x,y)$ defining the graph of $h$ is \textbf{a normal formula} if it is of the form $F(x)\wedge \psi(x,y)$ where
\begin{itemize}
\item $F(x)$ is a formula defining the set of definition of $h$,
\item $\psi(x,y)$ is a formula extended in $y$, conjunction of equalities between terms in $y$ and terms in $x$ (possibly constant).
\end{itemize}
\end{de}

\begin{lem}\label{prol} 
Let $h$ be an injection definable by a normal formula on a set $A\setminus B$ where $A$ and $B$ are definable subsets of $M$ and $B\subsetneq A$.
Then $h$ can be extended on $A$ by an injection $\tilde{h}$ definable by a normal formula.
\end{lem}

\begin{proof} 
Let $\phi(x,y)$ be a normal formula defining $h$.
We can write $\phi(x,y)$ as a conjunction of a formula $F(x)$ defining the set of definition of $h$ and, $\psi(x,y)$, a extended formula, conjunction of formulas of the form:
\begin{itemize}
\item $t(y)=t'(x)$ where $t(y)$ is a term in $y$ and $t'(x)$ a term in $x$,
\item $t(y)=c$ where $t(y)$ is a term in $y$ and $c\in M$.
\end{itemize}
According to lemma \ref{defsingl}, if there exists $x_0\in M$ \st $\psi(x_0,y)$ is a singleton, then for any\elmt $x\in M$, $\psi(x,y)$ defines a singleton. Thus this formula can indeed define the function which to $x\in M$ associates the only\elmt $y$ such that $\psi(x,y)$ is true.
\\
Let $G(x)$ be a formula defining $A$. The last point is to prove that $G(x)\wedge \psi(x,y)$ defines the graph of an injection.
\\
Since the formula, $\psi(x,y)$ is extended in $y$ it is also in $x$. There exists an\elmt $x_0$ (any $x_0\in A\setminus B$) such that the formula in $y$, $\psi(x_0,y)$ defines a singleton. According to lemma \ref{defsingl}, for any $x$, $\psi(x,y)$ defines a singleton: this precisely means that the function $\psi(x,y)$ is injective.
\end{proof}

\section{Simple sets}

Sets defined by positive primitive formulas -that we will later call "simple sets"- have two crucial properties that we prove in this section: 
\begin{itemize}
\item Every definable set is a Boolean combination of "simple sets"; 
\item Simple sets are in definable bijection with a Cartesian power of $M$ (with the convention that $M^0$ is a singleton).
\end{itemize}
These two properties imply that the Grothendieck ring of $M$ is a quotient of $\Z[T]$ where $T:=[M]$ and even of $\Z[T]/(T-T^2)$ since $M$ is in definable bijection with $M^2$.

The proof that $K_0(M)$ is precisely $\Z[T]/(T-T^2)$ will rely on another property of these sets: at fiwed depth, the topology they generate is noetherian. We develop these properties in this section.

\begin{de} 
Let $p\in \N^*$.
Let $\sqcup_{j=1}^n I_j$ be a partition of $\lbrace l, r \rbrace^{p} $ and
let $ C $ be a set of couples $ (i, c)$ where $i\in \lbrace g,
d \rbrace^{p} $
and $c\in M$. \\
We associate to $\sqcup_{j=1}^n I_j$ and $C$, the subset $A$ of $M$ whose\elmts $x$ are such as:
\begin{itemize}
 \item For every $1 \leq j \leq n $, and all $ j_1,
j_2 \in I_j$, $x_{j_1}=x_{j_2} $,
 \item for every $ (i, c) \in C $, $x_i=c $,
\item for every $ i, j $ such that there exists $ k $ with $ i, j \in I_k $, if $ (i, c) \in C $, then $ (j, c) \in C $ ($C$ is maximal).
\end{itemize}
The set $A$ is called \textbf{simple set of depth $p$} (the depth is
not specified when there is no ambiguity).
We call $\sqcup_{j=1}^n I_j$ the \textbf{partition} of $\lbrace l, r \rbrace^{p} $ adapted to $A$ and $ C $ its \textbf{set of constants}.
\end{de}

\begin{rem}\label{rem4} 
\begin{enumerate}
\item The depth of a set is not intrinsic to this set: any simple set of depth $p$ is also a simple set of depth $ q $ for every $ q> p$.
\item The intersection of two simple sets of depth $p$ is a simple set of depth $ p$.
\item If $ B_1, B, _2, B_3 $ are simple sets, not all identical, and of the same depth, then
$ B_1=B_2 \cup B_3 $ implies that $ B_2=\emptyset $ or $ B_3=\emptyset$.
\item If $ B_1 $ is a simple set and $ B_2 \subseteq B_1$ is a simple \ens strictly included in
$ B_1 $ and defined by a primitive positive formula of the same depth as $ B_1 $, then $ B_2$ is a simple set whose
partition refines the one adapted to $ B_1 $ and
whose set of constants contains the one of $ B_1$.
\end{enumerate}
\end{rem}

\begin{de} 
Let $q\in \N^* $ and $\sqcup_{i=1}^m J_i $ be a disjoint union of subsets included in $\lbrace l, r\rbrace^q$.\\
Consider the formula $\wedge_{i=1}^m (\wedge_{s, t \in J_i} x_s=x_t)$ which for every $i$, imposes equality between all $x_k $ for $ k \in J_i$. This formula possibly implies equalities between other terms. Consider the maximal subsets, $ I_i $, of $\lbrace l, r\rbrace^q $ such that for every $ k, l \in I_i $, the equality $x_k=x_l $ is implied by the formula
$\wedge_{i=1}^m  (\wedge_{s, t \in J_i} x_s=x_t)$.
\\
We call \textbf{partition} of $\lbrace l, r\rbrace^q $ \textbf{generated} by $\sqcup_i J_i $, the partition of $\lbrace l, r\rbrace^q $ made up of  $I_i $.

\medskip
\noindent
Let $\sqcup_i K_i $ be a partition of $\lbrace l, r\rbrace^q $ and $C$ a set of pairs $ (j, c)$ where $j \in \lbrace l, r\rbrace^q $ and $c\in M$.
The formula $[\wedge_{i=1}^m (\wedge_{s, t \in J_i} x_s=x_t)] \wedge [\wedge _{(j, c) \in C} x_j=c] $ defines a simple set whose set of constants, $C '$, is called \textbf{set of constants generated} by $C$ and $\sqcup_i K_i$.
\end{de}

\begin{lem} \label{noe} 
Let $p$ be an integer. Consider the topology on $M$ whose closed sets are generated by the simple sets of depth $p$.
Then this topology is Noetherian and the simple sets of depth  $p$ are its 
irreducible closed sets.
\end{lem}
\begin{proof} 
Let $F$ be the set of finite unions of simple set of depth $ p$.
\\
Let us check that the axioms defining closed sets are
satisfied by the elements of $ F$.
\\
$F$ is clearly stable by finite union.
\\
The set $M$ is the simple set of depth $p$ whose
set of constants is empty and whose partition only includes singletons. 
thus an element of $ F$.
\\
It is obvious that $F$ is stable by intersection.
\\
The fact that the simple sets are irreducibility closed sets of this topology is exactly the point 3 of the remark \ref{rem4}.
Let us now prove the Noetheriennity. Two simple sets are included one
in the other if, and only if:
\begin{itemize}
\item either the partition of the smallest refines the one of the largest,
\item or the set of
constants of the smallest contains the one of the largest.
\end{itemize}

Since a set of constants has at most $2^p$ elements and since there is a finite number of possible partitions, it cannot
exist infinite chains of simple sets.
\end{proof}

\begin{de}
Let $p\in\N$.
The topology whose sets of closed sets is generated by the simple sets of depth $p$ is denoted \textbf{$\mathcal{Top_p}$}.
\end{de}

\begin{lem} 
Every non-empty simple set is in definable bijection either with $M$ or a singleton.
\end{lem}

\begin{proof} 
Let $A$ be a non-empty simple set of $M$ of depth $p$, of partition $\sqcup_j I_j$, and of set of constants $C$.
\\
Let us consider the maximal subpartition $\sqcup_{i=1}^k I'_i$ of $\sqcup_j I_j$ such that for every $i$, the value of the terms corresponding to $I'_k$ are not determined by the constant set.
\\
If this maximal subpartition is empty then $A$ is reduced to a singleton.
\\
Otherwise, it is obvious that for every $a_1,\ldots, a_k\in M$, there is a unique element $x$ of $M$ such that 
\begin{itemize}
\item
for every $1\leq j\leq k$ and every $i\in I'_k$, $x_j=a_j$,
\item all the other terms of depth less or equal to $p$ have their value implied by $C$. 
\end{itemize}
The set
$A$ is hence in definable bijection with $M^k$. Since $M^2$ is in definable bijection with $M$, any Cartesian power of $M$ is in definable bijection with $M$ and so is $A$.
\end{proof}

\section{Decomposition of definable sets}

\begin{lem}\label{Unidec} 
Let $A$ be a set of depth less than $ p$.
Then $A$ is a Boolean combination of simple sets of
depth $ p$.
\end{lem}

\begin{proof} 
This is clear by elimination of quantifiers.
\end{proof}

\begin{de} 
Let $p\in \N$.
Let $A$ be a definable set and $\sqcup_i (B_i \setminus C_i)$ a partition of $A$ such that:
\begin{itemize}
\item for every $i$, $ B_i $ is a simple set of depth $p$,
\item for every $i$, $ C_i \subsetneq B_i $ is a union of simple sets of depth $ p$.
\end{itemize}
Then we say that $\sqcup_i (B_i \setminus C_i)$ is a \textbf{decomposition} of $A$ into simple sets of depth $ p$.
The sets $B_i$ are called \textbf{positive sets}, and the sets $C_i$, \textbf{negative sets}.
The sets $B_i \setminus  C_i$ are called \textbf{elementary sets}.
\end{de}

\begin{lem}\label{rrrrrr2} 
Let $p\in \N$ and $A\subseteq M^n$ be a definable set and $ A=\sqcup_{i=1}^n (B_i \setminus C_i)$ a decomposition in simple sets of depth $ p$.
Let us place ourselves in $\mathcal{Top}_p$. There exist $n'\in \N$ and for every $1+n\leq i\leq n+n'$, $B_i$, an irreductible closed set and $C'_i$, a closed set such that 
$ A=\sqcup_{i=1}^{n+n'} (B_i \setminus C_i)$ is a decomposition in simple set (of depth $p$) that satisfies that for every $1\leq i\leq n+n'$ there exist $1\leq i_1, \ldots, i_k \leq n+n'$ such that  $\cup_{m=1}^k B_{i_m}$ is the irreductibles decomposition of the closure of $C_i \cap A$.
\end{lem}

\begin{proof} 
Let $1\leq j\leq n$.
\\
We consider $B'_1, \ldots, B'_r$, the components of the irreducible decomposition of  the closure of $A \cap C_j$.
\\
Let $A^c$ be the complement set of $A$ in $M^n$.\\
We consider $C':=B_j \cap A^c$.
\\
We add, for every $1\leq j\leq r$, the elementary set $B'_j \setminus (C'\cap B'_j)$.\\
By construction, it is clear that $\sqcup_{j=1}^r B'_j \setminus (C'\cap B'_j)\subseteq A$.\\
\\
It is equally clear, since the topology is Noetherian, that we cannot repeat this operation indefinitely.
\\
By iterating this operation as many times as necessary, we obtain a decomposition satisfying the condition of the lemma.
\end{proof}

\begin{de}\label{single} 
A decomposition $ A=\sqcup_{j=1}^n (B_j \setminus C_j)$ which satisfies all the conditions of the lemma \ref{rrrrrr2} is said
\textbf{elementary}. We say that $ n$ is the \textbf{cardinality} of the partition.
\end{de}

\begin{lem}\label{evi} 
Let $\sqcup_{j=1}^n (B_j\setminus C_j)$ be an elementary decomposition of depth $p$. 
\\
Then, for every $1\leq j\leq n$, 
$C_j\cap A=\sqcup_{\lbrace i | B_i \subseteq C_j \rbrace} (B_i \setminus C_i)$ and 
\\
$B_j\cap A= (B_j\setminus C_j) \sqcup
\cup_{\lbrace i | B_i \subseteq C_j \rbrace} (B_i \setminus C_i)$.
\end{lem}

\begin{proof} 
For this lemma, we will use the following definition:
\begin{de} 
Let $\sqcup_{j=1}^n (B_j\setminus C_j)$ be a decomposition. 
\\
Let $1\leq j\leq n$
We say that $B_j$ is of \textbf{height} $h$ if $h$ is the largest number such that there exists $h$\elmts of $[|1;h|]$, $i_1,\ldots, i_p$ with $B_{i_1}\subsetneq \ldots\subsetneq B_{i_h}=B$.
\end{de}


\medskip
\noindent
Let $1\leq j\leq n$ and let $h$ be the height of $B_j$.
\\
Let us prove the lemma by induction on $h$.

\medskip
\noindent
Suppose that $h= 1$.
  \\
Since the decomposition is elementary, the irreducible decomposition of $\overline{C_j \cap A}$ is an union of positive sets.
Since $h= 1 $, this means that $\overline{C_j \cap A}=\emptyset$. The set
$ B_j \setminus C_j$ is equal to $ B_j \cap A $ and $ C_j \cap A=\emptyset=\sqcup _{\lbrace i~|~B_i \subseteq C_i \rbrace} (B_i \setminus C_i)$ since this last union is on the void.

\medskip
\noindent
Let $h\in \N^*$. Suppose the result shown for $h$. Let us show it for $h+ 1$.
\\
Since the decomposition is elementary, there exists $ i_1, \ldots, i_k $ such that
$\overline{C_j \cap A} $ admits $\sqcup_{m=1}^k B_{i_m} $ as irreducible decomposition closed sets (in the topology generated by the\sse of depth $p$).
\\
These sets are obviously of height less than $ B_j$.\\
By hypothesis of induction, for every
$1 \leq m \leq k $, $ C_{i_m} \cap A=\sqcup _{\lbrace i~|~B_i \subseteq C_{i_m} \rbrace} (B_i \setminus C_i)$ and $ B_{i_m} \cap A=(B_{i_m} \setminus C_{i_m}) \sqcup _{\lbrace i~|~B_i \subseteq C_{i_m} \rbrace} (B_i \setminus C_i)$.
\\
Since $\overline{C_j \cap A}=\cup_{m=1}^k B_{i_m} $, we have $ C_j \cap A=\cup_{m=1}^k (B_{i_m} \cap A)$ and therefore $ C_j \cap A=[\sqcup_{m=1}^k (B_{i_m} \setminus C_{i_m})] \sqcup \cup_{\lbrace i~|~B_i \subseteq C_{i_m} \rbrace} ( B_i \setminus C_i)$.
\\
This last union is obviously equal to $\sqcup _{\lbrace i~|~B_i \subseteq C_{i_m} \rbrace} (B_i \setminus C_i)$.
\\
Finally, since $ B_j \cap A=(B_j \setminus C_j) \sqcup (C_j \cap A)$, we have the second equality sought.
\\
This completes the induction.
\end{proof}

\section{Definable injection on simple set}

\begin{de}
Let $\Gamma$ be the graph of a definable function $h$. We say that $h$ is \textbf{normal} if $\Gamma$ is definable by a simple formula.
\end{de}
\begin{lem} \label{hhhhk2} 
Let $A$ be a simple set of depth $ p$.
Let $ h$ be an injection defined on $A$ whose graph is defined by a simple formula.
Then, there exists an integer $ q $ depending only on $h$ and $p$ such that for any simple set $X\subseteq A$, of depth $p$ the image of $X$ by $h$ is a simple set of depth $q$.
We say that the integer $q$ is \textbf{adapted to $A$ and $h$}.
\end{lem}

\begin{proof} 
Let $F(x, y):=(\wedge_j t_j (y)=t'_j(x)) \wedge (\wedge_j t '' _ j (y)=c_j)$ be a normal formula defining the graph of $ h$ on $ A $, and let $p'$ the maximum of depth in $y$ of all the
terms $ t_j (y)$ and $t'' _ j (y)$.
\\
We can always replace a formula of depth $p$ by an equivalent formula of depth $q$ for any $q\geq p$.\\
Up to considering a formula equivalent to $F$ and of depth at most $p'+p$ in $y$, we can assume that all the terms $ t'_j(x)$ are of depth $p$. Let $ q $ be the depth in $y$ of such a formula.
\\
Let $\sqcup_{i=1}^k I_i $ be the partition associated with $A$.
For every $i$ let $ J_i $ be the set of elements, $ u $, of $\lbrace l, r\rbrace^q $ such that $\exists k\in I_i$
$y_u=x_k $ is implied by $F(x, y)$.
\\
Let $ C $ be the set of couples; $ (i, c)$, with $c\in M$ and $i\in \lbrace l, r\rbrace^q $ such that $y_i=c $ is implied by $F(x, y)$.
\\
We denote $\sqcup_i \tilde{J} _i $ the partition of $\lbrace l, r\rbrace^q $ generated by $\sqcup_i J_i$ and  $ C $.
\\
Let $\tilde{C} $ be the set of constants generated by $ C $ and $\sqcup_i \tilde{J}_i$.
\\
The partition $\sqcup_i \tilde{J} _i $ and the set $\tilde{C} $ define a simple set of depth $q$ which is the image of $A$ by $ h$.
\end{proof}

\begin{lem}\label{7.2} 
Let $A$ be a definable set, $ h$ a definable injection from $A$ to itself.
Then there exists two integers $p, q$ such that
\begin{itemize}
\item $A=\sqcup_i (B_i \setminus C_i)$ is a decomposition of $A$ into simple sets of depth $p$,
\item for every $ j $, the restriction of $ h$ to $ B_j \setminus C_j$ can be extended on $ B_j$ by a definable injection $ h_j$,
\item for every $ j $, $ B'_j:=h_j (B_j)$ is a simple set of depth $ q $ and $ C'_j:=h_j (C_j)$ is a finite union of simple sets of depth $ q $,
\item for every $j$, the irreducible decomposition of $C_j$ in the topology generated by the simple sets of depth $p$ has the same number of\elmts as the irreducible decomposition of $C'_i$ in the topology generated by the simple sets of depth $q$.\\
Furthermore, the number of irreducible components of $C_i$ reduced to a singleton is equal to the number of components irreducible of $C'_i$ reduced to a singleton.
\end{itemize}
\end{lem}

\begin{proof} 
By elimination of the quantifiers, there exists an integer $p$ and a partition of $A:=\sqcup_i A_i$ in definable sets such that for every $i$, $A_i$ is a simple set of depth at most $p$ and $h$ restricted to $A_i$ is a normal function.\\
As every definable set of depth $p$ is a Boolean combination of simple sets of depth $p$, by decomposing the sets $A_i$, we get a finite partition of $A$, $ A=\sqcup_i (B_i \setminus C_i)$, where
\begin{itemize}
\item for every $ j $,
the set $B_j$ is a simple sets of depth $p$ and the set $C_j$ is a finite union of simple sets of depth $p$;
\item for every $ j $, the function $h$ restricted to $B_j \setminus C_j$ is a normal function.
\end{itemize}
According to lemma \ref{prol}, for every $ j $, $ h$ restricted to $ B_j \setminus C_j$ can be extended on $ B_j$ by a normal injection that we still denote $ h_j$.
\\
According to the previous lemmas, for every $ j $, there exists $q_j$ \st for any simple set $X$ of depth $p$ and included in $B_j$, $h_j(X)$ is a simple set of depth $q_j$.
Let $q:=\max_j (q_j)$.
As any simple set of a given depth 
can be considered as a simple set of greater depth, 
for every $ j $, $ B'_j:=h_j (B_j)$ can be considered as a simple set of depth $ q $ and $ C'_j:=h_j (C_j)$ is a finite union of simple sets of depth $ q $.
Let us write, for any integer $i$, the decomposition of $C_i$ in the topology generated by the\sse of depth $p$: $C_i=\cup_j B_{i,j}$. Since $h_i$ is injective, for every $j$, the\enss $h_i(B_{i,j})$ are distinct and are the components of the irreducible decomposition of $C'_i=h_i(C_i)$ in the topology generated by the simple sets of depth $q$.
\\
It is obvious, that the number of irreducible components of $C_i$ reduced to a singleton is equal to the number of irreducible components of $C'_i$ reduced to a singleton.
\end{proof}

\begin{de} 
Let $A$ be a definable set, and $ h$ a definable injection from $A$ to itself.
Let $ p, q $ be two integers. We say that $ A=\sqcup_i (B_i \setminus C_i)$, a finite decomposition of $A$ (which is not necessarily a partition) is \textbf{adapted} to $ h$ if it satisfies the properties of lemma \ref{7.2}.
\end{de}

\section{Representation of a decomposition: tree of a decomposition}

One way to represent a definable set is to associate a tree with one of its decompositions. A branch corresponds to a chain of simple sets that are included one in the other.

\begin{de}\label{arb} 
Let $ A $ be a definable set and
$ A = \sqcup_{j = 1}^h (B_j \setminus C_j) $ a decomposition. \\
Then we can associate to $A$ the tree $ T $ whose nodes are labelled by the surjection $ \phi $ from the nodes of $ T $ to the set $  \lbrace A \rbrace \sqcup \left \{B_1 \setminus C_1, \ldots,
	B_h \setminus C_h \right \} $ in the following way:
	\begin {itemize}
	\item If $ r $ is the root of $ T $, then $ \phi (r) = A $,
	\item If $ r $ is the root and $ B_{1}, \ldots, B_{m} $ are the maximal positive sets in the sense of inclusion, then $ r $ has exactly $ m $ children $ n_1, \ldots, n_m $, and furthermore,
	$ B_{1} \setminus C_{1} = \phi (n_1), \ldots, B_{m} \setminus C_{m} = \phi (n_m) $,
	\item Let $ n $ be a node of $ T $ not equal to its root and let $ 1 \leq
	j \leq h $ be such that $ B_j \setminus C_j = \phi (n) $. Let $ B_{j_1}, \ldots, B_{j_s} $ be
	the positive maximal sets (in the sense of inclusion) among those included in
	$ C_j $; then $ n $ has $ s $ children, $ n_1 \ldots, n_s $, and furthermore,
	$ B_{j_1} \setminus C_{j_1} = \phi (n_1), \ldots, B_{j_s} \setminus C_{j_s} = \phi (n_s) $.
	\end {itemize}
	We say that $ T $ is the \textbf{tree associated to the decomposition} $ \sqcup_{j = 1}^h
	(B_j \setminus C_j) = A $.
\end{de}

\begin{rem} 
a. We want to distinguish the root $A$ from the other nodes: even if the decomposition has only one element, $ B \setminus C $, the tree has a root whose image by $\phi$ is $A$ and a leaf whose image by $\phi$ is $ B \setminus C$.
\\
b. The function $\phi$ is not a bijection: some elementary sets can be associated with several nodes of the tree. All the nodes corresponding to the same set $ B \setminus C $, are roots of the same subtree.
\\
c. In the construction proposed in the previous demonstration,
the order in which trees $T_1, \ldots, T_s $ are planted on the root
is arbitrary. Nevertheless, this is the only phenomenon that could contradicts
the uniqueness of a tree associated with a given decomposition: its graphical representation is obviously not unique, but there is a single tree associated with a given decomposition.
\end{rem}

\begin{proof} 
Let $p$ be the length of a maximum chain formed by positive sets.
Let's prove the result by induction on $ p$.
\\
Suppose that $p= 1$. Then $ A=\sqcup_{j=1}^h (B_j \setminus C_j)$ where no negative set contains positive set. (Positive sets are the closed sets ones of the irreducible decomposition of $\overline{A}$.)
The associated tree is the tree of root $A$ of depth $1$ with $h$ leaves $f_1, \ldots, f_h$ corresponding to elementary sets.
To conclude we just have to set $\phi$ as
$\phi(r)=A $, $\phi(f_i)=B_i \setminus C_i$.

\medskip
\noindent
Let $p\in \mathbb{N}^*$.
Suppose the result shown for every $1 \leq i\leq h$, show it
for $p+ 1$.
\\
Up to renumbering the sets $ B_i $, we can assume that $ B_1, \ldots, B_s $ are the maximum positive sets (in the sense of inclusion) of $A$.
\\
We consider $T_0 $ the tree of root $A$ and whose $ s $ leaves $ f_1, \ldots, f_s $, correspond to $ B_1 \setminus C_1, \ldots, B_s \setminus C_s$. \\
We can apply the hypothesis of induction to $\sqcup_{j=s+1}^h (B_j \setminus C_j)$ to get the tree $T_1$.
\\
Each node of level 1 of $T_1 $ corresponds to a set $ B_j \setminus C_j$ where $ j $ ranges between $ s+1 $ and $ h $, such that, according to the lemma \ref{evi}, $ B_j$ is included in one of the sets $ B_1, \ldots, B_s$.
\\
Up to renumbering the sets, we can assume that the children of the root of $T_1$ are $B_{s+1}, \ldots, B_{s+t}$ where $s+t\leq h$.
 We consider the tree $T$ of root $A$ whose $ s $ child $ f_1, \ldots, f_s $, correspond to $ B_1 \setminus C_1, \ldots, B_s \setminus C_s $ and such that for every $1 \leq i\leq s $ and all $ j $ between $ s+1 $ and $ s+t $, we plant on the the node corresponding to $ B_i \setminus C_i $ the subtree of $T_1 $ of root $ B_j \setminus C_j$ if and only if, $ B_i \subsetneq B_j$. 
\\
It is obvious that $T$ satisfies the conditions of the statement.
\end{proof}

Here are some examples of trees associated with a decomposition:

\begin{ex} 
Let $ B_1, B_2, B_3 $ be\sse and $ C_1 $ a union of simple sets, such as:
\begin{itemize}
\item $ B_1 $ and $ B_2$ are not included in each other
\item $ C_1 \subsetneq B_1 \cap B_2$
\item $ B_3 \subsetneq C_1$.
\end{itemize}
Let us put $ C_2:=B_1 \cap B_2$.
Let $ A=(B_1 \setminus C_1) \sqcup (B_2 \setminus C_2) \sqcup B_3$.
\\
A tree associated with this decomposition of $A$ is:
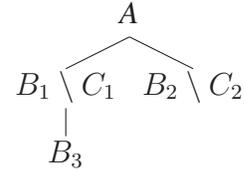
\Tree[.A [ $B_3$ ].{$B_1\setminus C_1$} {$B_2\setminus C_2$} ]

\end{ex}

\begin{ex} 
Let $ B_1, \ldots, B_3 $ simple sets such as:
\begin{itemize}
\item $ B_3 \subsetneq B_2 \subsetneq B_1$.
\end{itemize}

Let $ A=(B_1 \setminus B_2) \sqcup (B_2 \setminus B_3) \sqcup (B_3)$.

A tree associated with this decomposition of $A$ is:

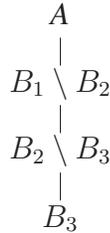
\Tree [.A [ [ $B_3$ ].{$B_2\setminus B_3$} ].{$B_1\setminus B_2$}  ]

\end{ex}

\begin{lem} \label{ajf} 
Let $p\in \N$. One places oneself in $T$, the topology generated by the\sse of depth $ p$.
Let $A $ be a closed sets  and $\sqcup_{i} (B_i \setminus C_i)=A$ an elementary decomposition of $A$ into simple sets of depth $ p$.
Let $T_0$ be a tree associated with this decomposition.
Then any node of $T_0$ corresponding to $ B_j \setminus C_j$ has as many children as $ C_j$ has irreducible components in $T$.
\end{lem}

\begin{proof} 
Since $A$ is a closed set, for every $ j $, $ C_j \cap A=\cup_i X_i $ where $\cup_i X_i $ is the irreducible decomposition of $ C_j$.
According to the lemma \ref{evi}, $ C_j \cap A $ is also equal to $\sqcup_{i | B_i \subsetneq C_j} (B_i \setminus C_i)$. By taking the closure of this last union, it follows that every set $X_i $ is equal to one of the set $ B_i $ : these correspond to the children of the node associated with $ B_j \setminus C_j$ which therefore has as many children as $ C_j$ has irreducible components (in $T$).
\end{proof}

\section{Computation of the Grothendieck ring}

\begin{lem}\label{jjj} 
Let $A$ be a definable set and $h$ a \dd injection from $A$ to $A$.\\
Let $A=\sqcup_i (B_i\setminus C_i)$ a decomposition adapted to $h$, and $p$ an integer such that all the\elmts of this decomposition are of depth $p$. Let $T$  be the the topology whose closed sets are generated by the simple sets of depth $p$. 
Then $A$ admits an elementary decomposition in $T$ adapted to $h$.
\end{lem}

\begin{proof} 
Let $j$ be an index and $X$ an irreducible component of $\overline{C_j\cap A}$.
Let us assume that $X$ is not one of the positives sets of the decomposition $A=\sqcup_i (B_i\setminus C_i)$. 
Let us show that there exists $C$ a closed set of $T$ such that $A=(\sqcup_i (B_i\setminus C_i))\sqcup (X\setminus C)$ is a decomposition adapted to $h$.
\\
Since the set $X$ is irreducible, there exists an index $i$ such that $X\subseteq B_i$ and $B_i$ is minimal for this property. The minimality of $B_i$ and the fact that $X$ is included in the closure of $A$ imply that $X$ is not included in $C_i$. Let $C:=X\cap C_i$.
The set $X\setminus C$ is then included in $B_i\setminus C_i$.
\\
It is enough to check that:
\begin{itemize}
\item $h$ is defined on $X\setminus C$ by a normal formula,
\item $h$ restricted to $X\setminus C$ can be extended on $X$ by an injection $\tilde{h}$ definable by the same normal formula than $h$ on $X\setminus C$,
\item $X':=\tilde{h}(X)$ is a simple\ens of depth $q$ and $C':=\tilde{h}(C)$ is an union of \ess of depth $q$.
\end{itemize}
Since $X\setminus C$ is included in $B_i\setminus C_i$, the two first properties are obvious. Let $\tilde{h}$ be the prolongation of $h$ on $X$ that satisfies the second property.
This prolongation is defined by a normal formula, $X':=\tilde{h}(X)$ is a \ess of depth $q$ and $C':=\tilde{h}(C)$ is an union of \ess of depth $q$.

\medskip
\noindent
The decomposition $A=(\sqcup_i (B_i\setminus C_i))\sqcup (X\setminus C)$ is thus adapted to $h$. By iterating this construction, we obtain an elementary decomposition adapted to $h$.
\end{proof}

\begin{cor} 
Let $A$ a definable set, and $h$ a definable injection from $A$ into $A$.
Then $A$ admits a decomposition adapted to $h$ (in the meaning of the previous demonstration) and elementary in the topology generated by the positive and negative sets of this decomposition.
\end{cor}

\begin{teo} 
Let $M$ be a model of the theory of the pairing function without cycles. The morphism of unitary rings of $\Z[X]$ sending $X$ to $[M]$ factorises as an isomorphism between $\Z[X]/(X-X^2)$ and $K_0(M)$. In particular, as $\Z[X]/(X-X^2)$ is itself isomorphic to $\Z^2$, the Grothendieck ring $K_0(M)$ is isomorphic to $\Z^2$.
\end{teo}

\begin{proof} 
Let $M$ be a model of the theory of the pairing functions without cycles.
Let $X$ be the class of $M$ in the Grothendieck ring. \\
It is obvious that $X=X^2$. \\
Let $A$ be a non-empty definable subset of $M$. Since $A$ is a Boolean combination of simple sets, and since every simple set is in definable bijection with $M$ or with a singleton, the class of $A$ in the \gr of
$M$ is an element of $\Z[X]/(X-X^2)$. 

\medskip
\noindent
Let us now consider $\Z[X]/(X-X^2)$ where $X$ is an undetermined variable.\\
Let us fix a family, $\mathbb{F}$, of finite sets containing for each $n\in \N$ exactly one set of cardinality $n$.\\
To each\elmts $aX+b \in \Z[X]/(X-X^2)$, with $a$ and $b$ positives, we associate the definable set $(M\times F)\sqcup G$ where $F$ is the element of $\mathbb{F}$ of cardinality $a$ and $G$ is the element of $\mathbb{F}$ of cardinality $b$.

\medskip
\noindent
Let us check that two distinct elements of this ring are associated with definable sets that do not have the same image in $K_0(M)$.
\\
In other words, let us show the following lemma:
\begin{lem}
Let $F, G, F'$ and $G'$ be finite sets. Let us assume that there exists a definable set $C$ disjoint from $(M\times F)\sqcup G$ and from $(M\times F')\sqcup G'$ such that $(M\times F)\sqcup G \sqcup C$ and $(M\times F')\sqcup G' \sqcup C$ are in definable bijection.
Then $|F|=|F'|$ and $|G|=|G'|$.
\end{lem}

\begin{proof}
Let $F, G, F', G'$ and $C$ be as in the lemma. 
We already know that $C$ is in definable bijection 
\end{proof}
Let us assume that there is such a bijection, that we will denote by $h$, between $(M\times F)\sqcup G$ and $(M\times F')\sqcup G'$.
We consider $\sqcup_{i=1}^N(B_i\setminus C_i)$ a decomposition in \ess adapted to $h$. We can furthermore assume that this decomposition is elementary.
We consider the tree $T$ associated to this decomposition as explained in \ref{arb}.

\medskip
\noindent
To each node $n_i$ of $T$ (distinct of the root), we associate $d_i$ the number of its children.
The total number of the nodes of $T$ is $N=\sum_i d_i + I$ where $I$ is the number of children of the root.

\medskip
\noindent
Since the decomposition is adapted to $h$,
$h(M)=\sqcup_{i=1}^N(B'_i\setminus C'_i)$ where for every $i$, $h(B_i\setminus C_i)=B'_i\setminus C'_i$.
Let $T'$ be a tree associated with this decomposition. To each node $n'_i$ of $T'$, we associate $d'_i$ the number of its children.
\\
Let us show that the number of children of the root of $T$ is equal to the number of children of the root of $T'$.
Since the total number of nodes of $T$ is obviously equal to the one of $T'$, $\sum_i d_i + I=\sum_i d'_i + I'$ where $I'$ is the number of children of the root of $T'$ and it is sufficient to verify that $\sum_i d'_i=\sum_i d_i$.

\medskip
\noindent
Since $(M\times F)\sqcup G$ is closed, and since the decomposition is elementary, according to the lemma \ref{ajf}, every node $n$ of $T$ corresponding to a set $B_i\setminus C_i$ has as much children as $C_i$ has irreducible components.
\\
According to the lemma \ref{7.2} and the fact that the decomposition is adapted to $h$, for every $i$, the number of irreducible components of $C'_i$ is equal to the number of irreducible components of $C_i$.
As the set $(M\times F')\sqcup G'$ is closed, the lemma \ref{ajf} implies that every node $n$ of $T'$ corresponding to a set $B'_i\setminus C'_i$ has as much children as $C'_i$ has irreducible components.
\\
The combination of these two points lead to $\sum_i d_i=\sum_i d'_i$ and thus to $I=I'$.

\medskip
\noindent
But the sets corresponding to the children of the root of $T$ (respectively $T'$) are the\elmts of the irreducible decomposition of $(M\times F)\sqcup G$ (respectively $h(M)$). As $(M\times F)\sqcup G$ has a irreducible decomposition of cardinality $|F|+|G|$ (respectively $|F'|+|G'|$), it thus follows that $|F|+|G|=|F'|+|G'|$.

\medskip
\noindent
Let 
\begin{itemize}
\item $A_1:=\lbrace i| B_i \text{ is a singleton }\rbrace$,
\item $A'_1:=\lbrace i| B'_i \text{ is a singleton }\rbrace$, 
\item $A_2:=\lbrace i| B_i \text{ is a singleton and there exists }j\neq i, B_i\subsetneq B_j \rbrace$,
\item $A'_2=\lbrace i| B'_i \text{ is a singleton and there exists }j\neq i, B'_i\subsetneq B'_j \rbrace$.
\end{itemize}

Let us notice that $|G|=|A_2|-|A_1|$ (respectively, $|G'|=|A'_2|-|A'_1|$).
The fact that $|A_1|=|A'_1|$ is obvious. Let us show that $|A_2|=|A'_2|$.
Because $(M\times F)\sqcup G$ is a closed set, the set $A_2$ is in bijection with the set $A_3:=\lbrace i| C_i \text{ admits a singleton in its irreductible components }\rbrace$. Similarly, $A'_2$ is in bijection with the set $A'_3:=\lbrace i| C'_i \text{ admits a singleton in its irreductible components }\rbrace$.
By the last point of lemma \ref{7.2}, $A_3$ and $A'_3$ have the same cardinality.
As a result, so do $A_2$ and $A'_2$ and hence $|G|=|G'|$.

\medskip  \noindent
It remains to show that $\Z[X]/(X-X^2)$ is isomorphic to $\Z^2$. This is a direct application of the Chinese remainder theorem:

$$\Z[X]/(X-X^2)\simeq \Z[X]/(X)\oplus \Z[X]/(1-X) \simeq \Z^2.$$

(An explicit isomorphism is for instance given by the homomorphism such that $1-X\mapsto (1, 0)$ and $X \mapsto (0,1)$.)

\end{proof}

\medskip
\noindent {\bf Acknowledgements.} 
I am very grateful to , Yves de Cornulier, Fran\c{c}oise Delon and David Marker for having read previous drafts and having made useful suggestions to improve them.

\end{document}